\documentclass[12pt]{amsart}

\usepackage{amssymb, a4wide, mathdots,url,hyperref}
\usepackage[all]{xy}

\providecommand{\germ}{\mathfrak}

\newcommand{\C}{\mathbb{C}}
\newcommand{\Z}{\mathbb{Z}}
\newcommand{\Q}{\mathbb{Q}}
\newcommand{\R}{\mathbb{R}}
\newcommand{\N}{\mathbb{N}}

\newcommand{\A}{\mathbb{A}}

\renewcommand{\Re}{\operatorname{Re}}

\newcommand{\St}{{\operatorname{St}}}
\newcommand{\cusp}{{\operatorname{cusp}}}
\newcommand{\sqr}{{\operatorname{sqr}}}
\newcommand{\temp}{{\operatorname{temp}}}
\newcommand{\gen}{{\operatorname{gen}}}
\newcommand{\ucusp}{{\operatorname{u-cusp}}}
\newcommand{\usqr}{{\operatorname{u-sqr}}}
\newcommand{\ugen}{{\operatorname{u-gen}}}

\renewcommand{\u}{{\operatorname{u}}}
\newcommand{\Irr}{\Pi}
\newcommand{\clss}{\Gamma}
\newcommand{\gal}{\sigma}
\newcommand{\inv}{\iota}
\newcommand{\oo}{\mathfrak{o}}

\newcommand{\Hom}{\operatorname{Hom}}

\newcommand{\bs}{\backslash}

\newcommand{\GL}{\operatorname{GL}}

\newcommand{\abs}[1]{\left|{#1}\right|}

\newcommand{\trivchar}{{\bf 1}}

\newcommand{\p}{{\mathfrak{p}}}   

\newtheorem{theorem}{Theorem}[section]
\newtheorem{lemma}[theorem]{Lemma}
\newtheorem{proposition}[theorem]{Proposition}

\newtheorem{definition}[theorem]{Definition}
\newtheorem{corollary}[theorem]{Corollary}
\newtheorem{example}[theorem]{Example}

\title {Gamma factors root numbers and distinction}
\author{Nadir Matringe}
\address{Laboratoire de Mathématiques et Applications
Téléport 2 - BP 30179, 11 Boulevard Marie et Pierre Curie, 86962 Futuroscope Chasseneuil Cedex, France}
\email{nadirmatringe@outlook.fr}
\author{Omer Offen}
\address{Department of Mathematics, Technion -- Israel Institute of Technology , Haifa 3200003, Israel}
\email{offen@tx.technion.ac.il}
\thanks{Omer Offen, partially supported by ISF grant No. 1394/12}
\date{\today}

\begin{document}

\setcounter{tocdepth}{1}

\begin{abstract}
We study a relation between distinction and special values of local invariants for representations of the general linear group over a quadratic extension of $p$-adic fields.
We show that the local Rankin-Selberg root number of any pair of distinguished representation is trivial and as a corollary we obtain an analogue for the global root number of any pair of distinguished cuspidal representations. We further study the extent to which the gamma factor at $1/2$ is trivial for distinguished representations as well as the converse problem.   
\end{abstract}

\maketitle
\section{Introduction}
This work continues to study the relations between distinguished representations and triviality at $1/2$ of the local gamma and epsilon factors in the context of $\GL_n$ over a quadratic extension.

Let $E/F$ be a quadratic extension of $p$-adic fields and fix a non-trivial character $\psi$ of $E$ that is trivial on $F$. A (smooth, complex valued) representation of $\GL_n(E)$ is called distinguished if it admits a non-zero $\GL_n(F)$-invariant linear form. 

In his dissertation \cite{MR2716711}, Ok showed that for $n>1$ and an irreducible, cuspidal representation $\pi$ of $\GL_n(E)$ with a central character that is trivial on $F^\times$ we have the following. The representation $\pi$ is distinguished if and only if for every irreducible, unitary, generic and distinguished representation $\pi'$ of $\GL_{n-1}(E)$ the local Rankin-Selberg gamma factor $\gamma(s,\pi,\pi',\psi)$ satisfies 
\[
\gamma(1/2,\pi,\pi',\psi)=1.
\]
Ok suggests that relations of this nature, may hold more generally. This was further explored by the second named author in \cite{MR2787356}, a paper that, unfortunately, contains some sloppy mistakes. The errors occur in the very last part  of the paper \cite[\S 7.2]{MR2787356} where local $L$-values are cancelled out at points where there may be poles. The upshot is that the main result, \cite[Theorem 0.1]{MR2787356}, is partially wrong and even the proof of the parts that are correct is invalid. The mistake was recently noted by the first named author. This led to the current collaboration, aiming to correct the mistakes and explore further the relation between distinction and special values of local factors. 
We remark that in addition, the proof of \cite[Corollary 7.2]{MR2787356} is not valid, although the statement is correct and a proof can be found in the recent work of Anandavardhanan and Matringe, \cite{AM}.

We now summarize our main results.
The first one says that the local Rankin-Selberg root number attached to a pair of distinguished representations is trivial. It was conjectured in \cite[Conjecture 5.1]{MR2448081}. More explicitly (see Theorem \ref{thm: eps triv}):
\begin{theorem}
Let $\pi_i$ be an irreducible and distinguished representation of $\GL_{n_i}(E)$, $i=1,2$. Then 
\[
\epsilon(1/2,\pi_1,\pi_2,\psi)=1.
\]
\end{theorem}
Together with an archimedean analogue carried out in \S \ref{sec: arch}, this gives triviality of the global  Rankin-Selberg root number of a pair of distinguished automorphic cuspidal representations. This is Theorem \ref{thm: glb main}. 

For the local gamma factors, we have the following (see Theorem \ref{thm: gamtriv} for a slightly more general statement):
\begin{theorem}
Let $\pi_i$ be an irreducible and distinguished representation of $\GL_{n_i}(E)$, $i=1,2$. If $\pi_1$ is unitary and generic and $\pi_2$ is tempered then 
\[
\gamma(1/2,\pi_1,\pi_2,\psi)=1.
\]
\end{theorem}
In Example \ref{example} we provide many pairs of distinguished representations for which gamma at $1/2$ is not trivial. (In fact, $\gamma(1/2,\pi_1,\pi_2,\psi)=-1$ if $\pi_i$ is the trivial representation of $\GL_i(E)$, $i=1,2$.)

Conversely, for essentially square-integrable representations, we show that the triviality of gamma at $1/2$ for enough twists by distinguished representations characterizes distinction. More explicitly, if $\delta$ is an irreducible, essentially square integrable representation of $\GL_n(E)$ and $\gamma(1/2,\delta,\pi,\psi)=1$ for any distinguished, irreducible representation $\pi$ of $\GL_m(E)$ for all $m\le n$ then $\delta$ is distinguished. 

In fact, only tempered twists are necessary and in many cases only for $m\le n-2$. Theorem \ref{thm qsr} is a generalization of the following.
\begin{theorem}
Let $\delta$ be an essentially square integrable representation of $\GL_n(E)$. If $\gamma(1/2,\delta,\pi,\psi)=1$
for every $m\le n$ and every irreducible, tempered and distinguished representation $\pi$ of $\GL_m(E)$ then $\delta$ is distinguished.
If $\delta$ is not of the form $\St_2(\rho)$ then taking $m\le n-2$ is enough. 
\end{theorem}
Here $\rho$ is cuspidal and the generalized Steinberg representation $\St_2(\rho)$ is the unique irreducible quotient of the parabolically induced representation $|\det|^{-1/2}\rho \times |\det|^{1/2}\rho$.

Example \ref{optimal} shows that we cannot expect such a converse theorem for more general distinguished representations. For essentially square integrable representations of the form $\St_2(\rho)$ we do expect twists up to $n-2$ to characterize distinction but a proof will require different methods. 

We remark that the above relative converse theorem is an analogue of the local converse theorem of Henniart \cite{MR1228128} or rather the improvement by Chen \cite{MR2257542} (see \cite[\S2]{MR3377749}).
We do not explore the extent to which it can be improved in the direction of \cite{1601.03656}.\\

Let us now outline the structure of the paper, and give the ideas behind some proofs. 

Theorem \ref{thm: eps triv}, which is the triviality 
of the non-archimedean Rankin-Selberg root number of a pair of irreducible distinguished representations is proved in Section \ref{sec: non-arch root}, using the cuspidal case due to Ok \cite{MR2716711}, and the classification of distinguished standard modules due to Gurevich \cite{MR3421655}. 

Section \ref{sec: central} is devoted to the proof of Proposition \ref{prop: trivial central}, which is the relative analogue of the well known fact that the gamma factors of twists by characters of an irreducible representation determine its central character \cite[Corollary 2.7]{MR3349305}. Our result says that if the gamma factors of the twists of an irreducible representation $\pi$ by distinguished characters are trivial at $1/2$, then the central character of $\pi$ is itself distinguished. Our proof is an adaptation of that in \cite{MR3349305}.

In Section \ref{sec: discrete converse}, we prove Theorem \ref{thm qsr}, our relative converse theorem for discrete series representations. 
Once again the idea is to reduce to Ok's cuspidal converse theorem, or rather its refinement due to Hakim and the second named author \cite{MR3377749}, 
using good twists and analytic methods. More precisely, twisting a generalised Steinberg representation $St_k(\rho)$ by representations of the form $|\det|^{-s}\rho^\vee \times |\det|^{s}\rho^\sigma$ (see Section \ref{notation} for undefined notation) allows to reduce to conjugate self-dual discrete 
series representations, and then use \cite{MR3377749}, where this case was taken care of. Then we observe that it is in fact enough 
to twist by such representations with $s\in i\R$, thanks to extension of meromorphic identities, and this allows us to twist only by distinguished tempered representations.

Finally, Theorem \ref{thm: arch rtnmb} in Section \ref{sec: complements}, which is the archimedean analogue of Theorem \ref{thm: eps triv} 
is a consequence of Kemarsky's classification of distinguished standard modules \cite{MR3404670}. The triviality of the global root numer of a pair of distinguished cuspidal automorphic representations then follows, and is Theorem \ref{thm: glb main}.

\section{Notation and Preliminaries}\label{notation}
Let $E$ be a $p$-adic field (a finite extension of $\Q_p$) and $\psi$ a non-trivial character of $E$. We denote by $\p_E=
\varpi_E\mathfrak{o}_E$ 
the maximal ideal of the ring of integers $\mathfrak{o}_E$ 
of $E$. Here $\varpi_E$ is a uniformizer of $E$. 
Let $G_n=\GL_n(E)$, and consider the involution $g\mapsto g^\inv:={}^tg^{-1}$ on $G_n$.

By a representation of $G_n$ we mean a complex valued, smooth representation. 
For a representation $\pi$ of $G_n$ let $\pi^\inv$ be the representation on the space of $\pi$ given by $\pi^\inv(g)=\pi(g^\inv)$ and let $\pi^\vee$ be the smooth dual of $\pi$. 
For representations $\pi_i$ of $G_{n_i}$, $i=1,\dots,k$ we denote by $\pi_1\times\cdots \times \pi_k$ the representation of $G_{n_1+\cdots +n_k}$ obtained from $\pi_1\otimes\cdots \otimes \pi_k$ by normalized parabolic induction.
For a representation $\pi$ and a character $\chi$ of $G_n$ let $\chi\pi$ be the representation on the space of $\pi$ given by $(\chi\pi)(g)=\chi(g)\pi(g)$. Let $\nu$ denote the character $\abs{\det}$ of $G_n$. 

Let $\Irr(n)$ be the collection of all irreducible representations of $G_n$. By convention, we also let $\Irr(0)$ be the set containing only the trivial representation of the trivial group $G_0$ and let $\Irr=\sqcup_{n=0}^\infty \Irr(n)$.

If $\pi\in\Irr$ then $\pi^\vee \simeq \pi^\inv$ \cite{MR0404534}. More generally, if $\pi_i\in \Irr$, $i=1,\dots, k$ then
\[
(\pi_1\times \cdots \times \pi_k)^\inv \simeq \pi_k^\vee \times\cdots\times \pi_1^\vee.
\]

Let $U_n$ be the group of upper triangular unipotent matrices in $G_n$ and, by abuse of notation, let $\psi$ also denote the character on $U_n$ defined by
\[
\psi(u)=\psi(u_{1,2}+\cdots +u_{n-1,n}).
\]

\begin{definition}
A representation $\pi$ of $G_n$ is of \emph{Whittaker type} if it is of finite length and 
\[
\dim \Hom_{U_n}(\pi,\psi)=1.
\] 
An irreducible representation of Whittaker type is called \emph{generic}.
\end{definition}

We use a lower index to denote special classes of representations in $\Irr$ as follows:
\begin{itemize}
\item $\Irr_{\cusp}$--cuspidal representations in $\Irr$;
\item $\Irr_{\sqr}$--essentially square integrable representations in $\Irr$;
\item $\Irr_{\temp}$--tempered representations in $\Irr$;
\item $\Irr_{\gen}$--generic representations in $\Irr$;
\item $\Irr_{\u-\bullet}$--unitary representations in the class $\Irr_{\bullet}$.
\end{itemize}
We have $\Irr_\ucusp \subseteq \Irr_\usqr \subseteq \Irr_\temp\subseteq \Irr_\ugen$ and $\Irr_\cusp \subseteq \Irr_\sqr \subseteq \Irr_\gen$.
Recall that 
\[
\Irr_\sqr=\{\St_k(\rho):\rho\in\Irr_\cusp,\,k\in\N\}
\]
where $\St_k(\rho)$ is the unique irreducible quotient of $\nu^{(1-k)/2}\rho\times \nu^{(3-k)/2}\rho\times\cdots\times\nu^{(k-1)/2}\rho$ (see \cite[Theorem 9.3]{MR584084}) and $\St_k(\rho)^\vee\simeq \St_k(\rho^\vee)$
(see \cite[Proposition 9.4]{MR584084}). 

For $\delta\in\Irr_\sqr$ let $e=e(\delta)\in \R$ be the unique real number such that $\nu^{-e}\delta\in\Irr_\usqr$. A representation of the form
\[
\lambda=\delta_1 \times \cdots \times \delta_k \ \ \ \text{where}\ \ \ \delta_1,\dots,\delta_k\in\Irr_\sqr, \ \ \ e(\delta_1)\ge \cdots\ge e(\delta_k)
\]
is called a standard module. It is of Whittaker type and admits a unique irreducible quotient $\pi$, the Langlands quotient of $\lambda$. The Langlands classification, is the bijection $\lambda\mapsto \pi$ from the set of all standard modules to $\Irr$ (see \cite{MR0507262}). We denote by $\lambda(\pi)$ the standard module with unique irreducible quotient $\pi\in\Irr$.
Note that
\[
\lambda(\pi^\vee)=\lambda(\pi)^\inv,\ \ \ \pi\in\Irr.
\]

For $\delta_1,\,\delta_2\in \Irr_\sqr$ we say that $\delta_1$ precedes $\delta_2$ and write $\delta_1\prec \delta_2$ if $\delta_i=\St_{k_i}(\rho_i)$, where $\rho_i\in \Irr_\cusp$, $i=1,2$ are such that
$\rho_2\simeq \nu^{(k_1+k_2)/2+1-i}\rho_1$ for some $i\in\{1,\dots,\min(k_1,k_2)\}$. 
By \cite[Theorem 9.7]{MR584084}, for $\delta_1,\dots,\delta_k\in\Irr_\sqr$ we have $\delta_1\times \cdots \times \delta_k\in\Irr$ if and only if $\delta_i\not\prec \delta_j$ for all $i\ne j$ and
\[
\Irr_\gen=\{\delta_1\times \cdots \times \delta_k: \delta_1,\dots,\delta_k\in\Irr_\sqr,\ \delta_i \not\prec \delta_j,\ 1\le i\ne j \le k\}.
\] 
More generally, if $\delta_1,\dots,\delta_k$ is ordered so that $\delta_i\not\prec \delta_j$ for all $1\le i< j \le k$ then $\lambda=\delta_1 \times \cdots\times \delta_k$ is a standard module, independent of any such order.
We will say that such a realization $\delta_1\times \cdots\times \delta_k$ of $\lambda$ is in standard form.

For representations $\pi$ of $G_n$ and $\tau$ of $G_m$ of Whittaker type let $L(s,\pi,\tau)$ and $\epsilon(s,\pi,\tau,\psi)$ be the local Rankin-Selberg  $L$ and $\epsilon$-factors defined by \cite{MR701565} and let 
\[
\gamma(s,\pi,\tau,\psi)=\epsilon(s,\pi,\tau,\psi)\frac{L(1-s,\pi^\inv,\tau^\inv)}{L(s,\pi,\tau)}.
\]

For any $\pi,\,\tau\in\Irr$ the local $L$, $\epsilon$ and $\gamma$-factors for the pair $(\pi,\tau)$ are defined as the corresponding factors for the pair $(\lambda(\pi),\lambda(\tau))$. In particular, we have
\begin{equation}\label{eq: gamma}
\gamma(s,\pi,\tau,\psi)=\epsilon(s,\pi,\tau,\psi)\frac{L(1-s,\pi^\vee,\tau^\vee)}{L(s,\pi,\tau)}.
\end{equation}

We list below, some basic properties of the local rankin-selberg $L$ and $\epsilon$ factors. They are immediate from the definitions and will be used freely in the sequel. Let $\pi$ and $\tau$ be either in $\Irr$ or of Whittaker type. Then
\begin{enumerate}
\item $L(s,\pi,\tau)=L(s,\tau,\pi)$ and $\epsilon(s,\pi,\tau,\psi)=\epsilon(s,\tau,\pi,\psi)$.
\item $L(s,\nu^t\pi,\tau)=L(s+t,\pi,\tau)$ and $\epsilon(s,\nu^t\pi,\tau,\psi)=\epsilon(s+t,\pi,\tau,\psi)$.
\item $\gamma(s,\pi,\tau,\psi)\gamma(1-s,\pi^\inv,\tau^\inv,\psi^{-1})=1=\epsilon(s,\pi,\tau,\psi)\epsilon(1-s,\pi^\inv,\tau^\inv,\psi^{-1})$.
\end{enumerate}

The following is a consequence of \cite[Proposition 8.1]{MR701565}.
\begin{lemma}\label{lem: cusp L}
Let $\rho_1,\,\rho_2\in \Irr_\cusp$. Then $L(s,\rho_1,\rho_2^\vee)=1$ unless $\rho_2=\nu^u\rho_1$ for some $u\in \C$. More generally, the meromorphic function $L(s,\rho_1,\rho_2^\vee)$  has at most simple poles. It has a pole at $s=u$ if and only if $\rho_2\simeq\nu^u\rho_1$.
In particular, if $\rho_1,\,\rho_2\in \Irr_{\ucusp}$ then $L(s,\rho_1,\rho_2^\vee)$ is holomorphic whenever $\Re(s)\ne 0$. \qed
\end{lemma}

\begin{lemma}\label{lem: sqr L}
Let $\delta_1,\,\delta_2\in \Irr_\sqr$. Then $L(s,\delta_1,\delta_2^\vee)$ has at most simple poles. It has a pole at $s=u$ if and only if $\delta_1\prec \nu^{1-u}\delta_2$. In particular, if $\delta_i\in\Irr_\usqr$, $i=1,2$ then $L(s,\delta_1,\delta_2^\vee)$ is holomorphic whenever $\Re(s)>0$.
\end{lemma}
\begin{proof}
Write $\delta_i=\St_{k_i}(\rho_i)$ for $\rho_i\in\Irr_\cusp$ and $k_i\in \N$, $i=1,2$. It follows from \cite[Proposition 8.2]{MR701565} (and Lemma \ref{lem: cusp L}) that
\[
L(s,\delta_1,\delta_2^\vee)=\prod_{i=1}^{\min(k_1,k_2)} L(s+\frac{k_1+k_2}2-i,\rho_1,\rho_2^\vee).
\]
It follows from Lemma \ref{lem: cusp L} that at most one factor on the right hand side can have at most a simple pole at $s=u$ and that this is the case, if and only if 
\[
\rho_2\simeq\nu^{u+\frac{k_1+k_2}2-i}\rho_1
\]
for some $i=1,\dots,\min(k_1,k_2)$. Equivalently, if and only if
\[
\nu^{1-u}\rho_2\simeq\nu^{\frac{k_1+k_2}2+1-i}\rho_1
\]
for some $i=1,\dots,\min(k_1,k_2)$. 

Since $\nu^{1-u}\St_{k_2}(\rho_2)\simeq \St_{k_2}(\nu^{1-u}\rho_2)$, this is equivalent to the condition $\delta_1\prec \nu^{1-u}\delta_2$.
\end{proof}

\begin{lemma}\label{lem: mult} We have the following multiplicative properties.
\begin{enumerate}
\item \cite[Theorem 3.1]{MR701565} \label{part: gamma} Let $\pi=\pi_1\times\cdots \times \pi_k$ and $\tau$ be of Whittaker type. Then 
\[
\gamma(s,\pi,\tau,\psi)=\prod_{i=1}^k \gamma(s,\pi_i,\tau,\psi).
\] 
\item \cite[Proposition 9.4]{MR701565} \label{part: Leps} Let $\lambda=\delta_1\times\cdots \times \delta_k$  and $\xi=\delta_1' \times \cdots\times \delta_l'$ be standard modules in standard form. Then
\[
L(s,\lambda,\xi)=\prod_{i=1}^k \prod_{j=1}^l L(s,\delta_i,\delta_j')
\]
and consequently,
\[
\epsilon(s,\lambda,\xi,\psi)=\prod_{i=1}^k \prod_{j=1}^l \epsilon(s,\delta_i,\delta_j',\psi).
\] \qed
\end{enumerate}

\end{lemma}

\begin{lemma}\label{lem: cuspsupp}
For $\pi,\,\tau\in \Irr$, the $\gamma$-factor $\gamma(s,\pi,\tau,\psi)$ depends only on the cuspidal support of $\pi$ and $\tau$. More explicitly,
if $\tau_1,\dots,\tau_k\in \Irr_\cusp$ are such that $\tau\subseteq \tau_1\times\cdots \times \tau_k$ then
\[
\gamma(s,\pi,\tau,\psi)=\prod_{i=1}^k \gamma(s,\pi,\tau_i,\psi).
\]
\end{lemma}
\begin{proof}
By passing to the standard modules $\lambda(\pi)$ and $\lambda(\tau)$ and applying Lemma \ref{lem: mult} \eqref{part: gamma}  we may assume that $\tau\in \Irr_\sqr$.
In that case, write $\tau=\St_k(\rho)$ for $\rho\in\Irr_\cusp$. Then, the Whittaker space for $\tau$ is contained in that for  $\nu^{(k-1)/2}\rho\times \cdots\times \nu^{(1-k)/2}\rho$. Since the gamma factor is a quotient of Rankin-Selberg integrals for any choice of Whittaker functions that provide non-zero integrals,  $\gamma(s,\pi,\tau,\psi)=\gamma(s,\pi,\nu^{(k-1)/2}\rho\times \cdots\times \nu^{(1-k)/2}\rho,\psi)$. The Lemma now follows by applying Lemma \ref{lem: mult} \eqref{part: gamma}  to the right hand side.
\end{proof}

\subsection{Distinguished representations}

Assume from now on that $F$ is a $p$-adic field, $E/F$ a quadratic extension and $\psi|_F=1$. Let $\gal$ denote the associated Galois action. In particular, we have $\psi^\gal=\psi^{-1}$.
It is further obvious from the definitions of the local factors that for $\pi$ and $\tau$ either in $\Irr$ or of Whittaker type we have 
\begin{equation}\label{eq: Leps}
L(s,\pi,\tau)=L(s,\pi^\gal,\tau^\gal) \ \ \ \text{and}\ \ \ \epsilon(s,\pi,\tau,\psi)= \epsilon(s,\pi^\gal,\tau^\gal,\psi^\gal). 
\end{equation}

Let $\eta_{E/F}$ be the order two character of $F^\times$ that is trivial on norms from $E^\times$ to $F^\times$ and fix a character $\eta$ of $E^\times$ extending $\eta_{E/F}$.
Let $H_n=\GL_n(F)$. 
\begin{definition}
A representation $\pi$ of $G_n$ is \emph{distinguished} if
\[
\Hom_{H_n}(\pi,\trivchar)\ne 0
\]
and \emph{$\eta_{E/F}$-distinguished} if
\[
\Hom_{H_n}(\pi,\eta_{E/F}\circ \det)\ne 0.
\]
\end{definition}
Clearly, $\pi$ is distinguished if and only if $\eta\pi$ is $\eta_{E/F}$-distinguished.
Flicker showed the following:
\begin{lemma}\label{lem: flick}  \cite[Propositions 11 and 12]{MR1111204}
For $\pi\in\Irr(n)$ we have
\begin{enumerate}
\item $\dim\Hom_{H_n}(\pi,\trivchar)\le 1$.
\item \label{part: sym} If $\pi$ is distinguished then $\pi^\gal\simeq\pi^\vee$.
\end{enumerate} \qed
\end{lemma}

We record here some results regarding distinguished representaitons.
\begin{proposition}\label{prop: dist}
\begin{enumerate}
\item (\cite[Corollary 1.6]{MR2063106} and \cite{MR2075482})   Let $\delta\in\Irr_\sqr$ be such that $\delta^\gal\simeq\delta^\vee$. Then $\delta$ is either distinguished or $\eta_{E/F}$-distinguished but not both.
\item (\cite[Theorem 1.3]{MR2448081}, see \cite[\S 4.4]{MR2146859} or \cite[Corollary 4.2]{MR2500974}) Let $\delta=\St_k(\rho)\in\Irr_\sqr$. Then $\delta$ is distinguished if and only if one of the following two conditions hold
\begin{itemize}
\item $\rho$ is distinguished and $k$ is odd 
\item $\rho$ is $\eta_{E/F}$-distinguished and $k$ is even.
\end{itemize}
\end{enumerate}
\end{proposition}

\begin{proposition}\label{prop: dist std}
Let $\lambda$ be a standard module. Then $\lambda$ is distinguished if and only if
there exist $\delta_i\in \Irr_\sqr$, $i=1,\dots,t$ with $e(\delta_1)\ge \cdots\ge e(\delta_t)\ge 0$ and distinguished $\tau_j\in \Irr_\sqr$, $j=1,\dots,s$ such that
\begin{equation}\label{eq: dist std}
\lambda\simeq \delta_1^\gal\times \cdots \times \delta_t^\gal \times \tau_1\times\cdots\times \tau_s\times\delta_t^\vee\times\cdots\times\delta_1^\vee. 
\end{equation}  
\end{proposition}
\begin{proof}
If $\lambda$ is distinguished then it follows from \cite[Proposition 3.4]{MR3421655} that $\lambda$ can be realized as $\lambda=\delta_1'\times \cdots \times \delta_k'$ in standard form, so that there exists an involution $w\in S_k$ such that $(\delta_i')^\gal \simeq (\delta_{w(i)}')^\vee$ for all $i=1,\dots,k$ and $\delta_i'$ is distinguished whenever $w(i)=i$. Let $k=s+2t$ where $s=\abs{\{i: w(i)=i\}}$, $\{i_1,\dots,i_s\}=\{i: w(i)=i\}$ and $(\tau_1,\dots,\tau_s)=(\delta_{i_1}',\dots,\delta_{i_s}')$. Let $\{j_1,\dots,j_t\}=\{i: w(i)\ne i \text{ and either } e(\delta_i')>0 \text{ or } e(\delta_i')=0 \text{ and }i<w(i)\}$ be ordered in such a way that $e(\delta_{j_1}')\ge \cdots \ge e(\delta_{j_t}')$ and set $(\delta_1,\dots,\delta_t)=((\delta_{j_1}')^\gal,\dots,(\delta_{j_t}')^\gal)$. Then \eqref{eq: dist std} holds by the independence of the standard module on a standard form realization.

Assume now that $\lambda$ is of the form \eqref{eq: dist std}. By a standard, closed orbit contribution, argument $\tau_1\times\cdots\times\tau_s$ is distinguished (see e.g. \cite[Lemma 6.4]{MR2787356}) and it therefore follows from \cite[Proposition 2.3]{MR3421655} that $\lambda$ is distinguished.
\end{proof}

\section{The local root number of a distinguished representation}\label{sec: non-arch root}

We first record straight-forward consequences of the basic properties of $L$ and $\epsilon$ factors for representations that are isomorphic to the Galois twist of their smooth dual.
\begin{lemma}\label{lem: L prop}
Let $\pi,\,\tau\in \Irr$ satisfy $\pi^\gal\simeq \pi^\vee$ and $\tau^\gal\simeq \tau^\vee$. Then we have
\begin{enumerate}
\item \label{part: vee} $L(s,\pi,\tau)=L(s,\pi^\vee,\tau^\vee)$.
\item \label{part: gamhalf} Therefore, if $L(s,\pi,\tau)$ is holomorphic at $s=1/2$ then $\gamma(1/2,\pi,\tau,\psi)=\epsilon(1/2,\pi,\tau,\psi)$.
\item \label{part: gamsqr} $\gamma(s,\pi,\tau,\psi)\gamma(1-s,\pi,\tau,\psi)=1=\epsilon(s,\pi,\tau,\psi)\epsilon(1-s,\pi,\tau,\psi)$.
\item \label{part: epsqr} In particular, $\epsilon(1/2,\pi,\tau,\psi)^2=1$. 
\end{enumerate}
\end{lemma}

We recall the following result from Ok's thesis \cite{MR2716711}.
\begin{lemma}[Ok]\label{lem: Ok}
If $\rho_1,\rho_2\in\Irr_\cusp$ are both distinguished then $\gamma(1/2,\rho_1,\rho_2,\psi)=1$. \qed
\end{lemma}
We therefore also have
\begin{corollary}\label{cor: Ok}
If $\rho_1,\rho_2\in\Irr_\cusp$ are both distinguished then $\epsilon(1/2,\rho_1,\rho_2,\psi)=1$.
\end{corollary}
\begin{proof}
It follows from Lemma \ref{lem: flick} \eqref{part: sym} that $\rho_i^\gal\simeq\rho_i^\vee$ and in particular that $\rho_i\in\Irr_{\ucusp}$, $i=1,2$. 
It therefore follows from Lemma \ref{lem: cusp L} that $L(s,\rho_1,\rho_2)$ is holomorphic at $s=1/2$ and therefore from  Lemma  \ref{lem: L prop} \eqref{part: gamhalf} that
$\epsilon(1/2,,\rho_1,\rho_2,\psi)=\gamma(1/2,\rho_1,\rho_2,\psi)$.
The corollary now follows from Lemma \ref{lem: Ok}.
\end{proof}
\begin{lemma}\label{lem: conj dual}
Let $\pi$ be either in $\Irr$ or of Whittaker type and satisfy $\pi^\gal\simeq\pi^\inv$ and let $\delta\in\Irr_\sqr$ be such that $e(\delta)\ge 0$. Then 
\[
\epsilon(1/2,\pi,\delta^\gal\times \delta^\vee,\psi)=1.
\]
\end{lemma}
\begin{proof}
It follows from Lemma  \ref{lem: mult} that
\[
\epsilon(1/2,\pi,\delta^\gal\times \delta^\vee,\psi)=\epsilon(1/2,\pi,\delta^\gal,\psi)\epsilon(1/2,\pi,\delta^\vee,\psi).
\]
Applying \eqref{eq: Leps} we have
\[
\epsilon(1/2,\pi,\delta^\gal,\psi)=\epsilon(1/2,\pi^\inv,\delta,\psi^{-1})=\epsilon(1/2,\pi,\delta^\vee,\psi)^{-1}
\]
and the lemma follows.
\end{proof}

\begin{lemma}\label{lem twosq}
Let $\delta_1,\delta_2\in\Irr_\sqr$ be such that $\delta_i^\gal\simeq\delta_i^\vee$, $i=1,2$. Then
\begin{enumerate}
\item $\gamma(1/2,\delta_1,\delta_2,\psi)=\epsilon(1/2,\delta_1,\delta_2,\psi)$;
\item If $\delta_1$ and $\delta_2$ are distinguished then $\gamma(1/2,\delta_1,\delta_2,\psi)=1$;
\end{enumerate}
\end{lemma}
\begin{proof}
Note that by the given symmetry $\delta_1,\,\delta_2\in\Irr_\usqr$. The first equality therefore follows from Lemma \ref{lem: L prop} \eqref{part: gamhalf} and Lemma \ref{lem: sqr L}.
Write $\delta_i=\St_{k_i}(\rho_i)$, $i=1,2$, where $\rho_i\in\Irr_\ucusp$ and $k_i\in\N$. It follows from Lemma \ref{lem: cuspsupp} that
\begin{equation}\label{eq: exp gam}
\gamma(s,\delta_1,\delta_2,\psi)=\prod_{i_1=1}^{k_1}\prod_{i_2=1}^{k_2} \gamma( s+\frac{k_1+k_2}2+1-i_1-i_2,\rho_1,\rho_2,\psi).
\end{equation}
By Lemma \ref{lem: cusp L} for $u\in \R$ the function $\gamma( s,\rho_1,\rho_2,\psi)$ is holomorphic and non-zero at $s=u$ unless $\rho_2\simeq\rho_1^\vee$ and $u\in \{0,1\}$.
By Lemma \ref{lem: L prop} \eqref{part: gamsqr} we therefore have that
\begin{multline}\label{eq: pair one}
\gamma( 1/2+\frac{k_1+k_2}2+1-i_1-i_2,\rho_1,\rho_2,\psi) \times \\ \gamma( 1/2+\frac{k_1+k_2}2+1-(k_1+1-i_1)-(k_2+1-i_2),\rho_1,\rho_2,\psi)=1
\end{multline}
whenever either $\rho_2\not\simeq \rho_1^\vee$ or $(k_1+k_2+3)/2-i_1-i_2\not\in\{0,1\}$.

If $\delta_1$ and $\delta_2$ are distinguished then it follows from Proposition \ref{prop: dist} that $k_1\equiv k_2\mod 2$ and that $\rho_i$ is distinguished if and only if $k_i\equiv 1\mod 2$, $i=1,2$. Therefore, all the terms on the right hand side of \eqref{eq: exp gam} cancel out in pairs except for the term $\gamma(1/2,\rho_1,\rho_2,\psi)$ that occurs only if $k_1\equiv k_2\equiv 1\mod 2$. In addition, if this is the case then $\gamma(1/2,\rho_1,\rho_2,\psi)=1$ by  Lemma \ref{lem: Ok}. The lemma follows.
\end{proof}

We are now ready to prove \cite[Conjecture 5.1]{MR2448081}.
\begin{theorem}\label{thm: eps triv}
Let $\pi,\,\tau\in \Irr$ both be distinguished. Then 
\[
\epsilon(\frac12,\pi,\tau,\psi)=1.
\]
\end{theorem}
\begin{proof}
If $\pi\in \Irr$ is distinguished, then $\lambda(\pi)$ is also distinguished and by definition $\epsilon(\frac12,\pi,\tau,\psi)=\epsilon(\frac12,\lambda(\pi),\lambda(\tau),\psi)$. By Lemma \ref{lem: mult} and Proposition \ref{prop: dist std} it is enough to show that:
\begin{itemize}
\item $\epsilon(1/2,\delta_1,\delta_2,\psi)=1$ for every distinguished $\delta_1,\,\delta_2\in \Irr_\sqr$;
\item $\epsilon(1/2,\delta_1,\delta_2^\gal \times\delta_2^\vee,\psi)=1$ for every $\delta_1,\,\delta_2\in \Irr_\sqr$ with $\delta_1$ distinguished and $e(\delta_2)\ge 0$;
\item $\epsilon(1/2,\delta_1^\gal\times\delta_1^\vee,\delta_2^\gal \times\delta_2^\vee,\psi)=1$ for every $\delta_1,\,\delta_2\in \Irr_\sqr$ with $e(\delta_i)\ge 0$, $i=1,2$.
\end{itemize}
The first equality follows from Lemma \ref{lem twosq}. The last two equalities follow from Lemma \ref{lem: conj dual}.
\end{proof}

Let 
\[
\Irr_{<\frac12}=\{\delta_1\times\cdots\times\delta_k:\delta_i\in\Irr_\sqr,\ \abs{e(\delta_i)}<1/2,i=1,\dots,k\}.
\]
Then $\Irr_\ugen \subseteq \Irr_{<\frac12}\subseteq \Irr_\gen$.
\begin{theorem}\label{thm: gamtriv}
Let $\pi,\,\tau\in \Irr$ both be distinguished. If $\pi\in\Irr_{<\frac12}$ and $\tau\in \Irr_\temp$, then 
\[
\gamma(1/2,\pi,\tau,\psi)=1.
\]
\end{theorem}
\begin{proof}
If $\pi\in\Irr_{<\frac12}$ and $\tau\in \Irr_\temp$ then it follows from Lemmas \ref{lem: sqr L} and \ref{lem: mult} that $L(s,\pi,\tau)$ is holomorphic at $s=1/2$, from Lemmas \ref{lem: flick} \eqref{part: sym} and \ref{lem: L prop} \eqref{part: gamhalf} that $\gamma(1/2,\pi,\tau,\psi)=\epsilon(1/2,\pi,\tau,\psi)$ and from Theorem \ref{thm: eps triv} that $\epsilon(1/2,\pi,\tau,\psi)=1$. The theorem follows.


\end{proof}

\begin{example}\label{example}
Let $\rho_0\in\Irr_\ucusp$ be such that $\rho_0^\gal\not\simeq\rho_0^\vee$ and set $\rho=\nu^\alpha\rho_0$ for some $\alpha\in\R$. Let $\pi=\nu^{-1/2}\rho^\vee \times \nu^{1/2}\rho^\gal$ and $\tau=\rho\times (\rho^\gal)^\vee$. Then $\pi,\,\tau\in \Irr_\gen$ are both distinguished. It follows as in the above proof that $\epsilon(1/2,\pi,\tau,\psi)=1$,
\[
\gamma(1/2,\pi,\tau,\psi)=\lim_{s\to 1/2}\frac{L(1-s,\pi,\tau)}{L(s,\pi,\tau)}
\]
 and
 \[
 L(s,\pi,\tau)=L(s-1/2,\rho^\vee,\rho)L(s-1/2,\rho^\vee,(\rho^\gal)^\vee)L(s+1/2,\rho^\gal,\rho)L(s+1/2,\rho^\gal,(\rho^\gal)^\vee).
 \]
By Lemma \ref{lem: sqr L} we have that $L(s-1/2,\rho^\vee,\rho)$ has a simple pole at $s=1/2$ while $L(s-1/2,\rho^\vee,(\rho^\gal)^\vee)L(s+1/2,\rho^\gal,\rho)L(s+1/2,\rho^\gal,(\rho^\gal)^\vee)=L(s+1/2,\rho^\gal,(\rho^\gal)^\vee)$ is holomorphic at $s=1/2$ and therefore  
\[
\lim_{s\to 1/2}\frac{L(1-s,\pi,\tau)}{L(s,\pi,\tau)}=-1.
\]
This gives examples of distinguished $\pi,\,\tau\in\Irr$ such that $\gamma(1/2,\pi,\tau,\psi)=-1$. Note that if $\alpha=0$ then $\pi\in \Irr_\gen$ and $\tau\in \Irr_\temp$ while if $-1/2<\alpha<0$ then $\pi,\,\tau\in\Irr_{<\frac12}$.

For an example where $\pi,\,\tau\in\Irr_\ugen$, let $\rho\in \Irr_\cusp$ be such that $\rho^\gal\simeq\rho^\vee$ and let $0<\alpha,\beta<1/2$ be such that $\alpha+\beta=1/2$. Set $\pi=\nu^\alpha\rho\times\nu^{-\alpha}\rho$ and $\tau=\nu^\beta \rho^\vee\times \nu^{-\beta}\rho^\vee$ and note that $\pi,\,\tau\in\Irr_\ugen$ are both distinguished. Also
\[
L(s,\pi,\tau)=L(s+\alpha+\beta,\rho,\rho^\vee)L(s+\alpha-\beta,\rho,\rho^\vee)L(s-\alpha+\beta,\rho,\rho^\vee)L(s-\alpha-\beta,\rho,\rho^\vee).
\]
The first three terms on the right hand side are holomorphic at $s=1/2$ while the last has a simple pole at $s=1/2$. Therefore,
\[
\gamma(1/2,\pi,\tau,\psi)=\lim_{s\to 1/2}\frac{L(1-s,\pi,\tau)}{L(s,\pi,\tau)}=-1.
\]
In fact, one can even take $\tau\in \Irr_\usqr$, for example if $\rho$ is cuspidal and distinguished, choose $\tau=St_3(\rho)$ and 
$\pi=\nu^{-3/2}\rho\times \nu^{3/2}\rho$.
 \end{example}

\section{A local relative converse theorem for central characters}\label{sec: central}

The purpose of this section is to show that if the gamma factor at $1/2$ is trivial for twists by all distinguished characters of $E^\times$ then the central character is distinguished.
We begin with some simple observations about additive and multiplicative characters.

\begin{lemma}\label{inclusions}
For every integer $m$, one has 
\begin{itemize}
\item $F+\p_E^{m} \subsetneq F+\p_E^{m-1}$ if $E$ is unramified over $F$, whereas
\item $F+\p_E^{2m+1}=F+\p_E^{2m} \subsetneq F+\p_E^{2m-1}$ if $E$ is ramified over $F$.
\end{itemize}
\end{lemma}
\begin{proof}
Note that $\varpi_F^k(F+\p_E^m)=F+\p_E^{m+fk}$ where $f$ is the degree of the residual field extension for $E/F$. This shows that if the statement is true for some $m_0$ then it is true for all $m$. 

If $F+\p_E=F+\oo_E$ then for $u\in \oo_E^\times$ write $u=x+y$ with $x\in F$ and $y\in\p_E$. Then, clearly $x\in \oo_F^\times$ and therefore
$u=x(1+x^{-1}y)\in\oo_F^\times(1+\p_E)$. This shows that  $\oo_E^\times=\oo_F^\times(1+\p_E)$, i.e., that $E/F$ is ramified. The case where $E/F$ is unramified follows. 

Assume now that $E/F$ is ramified. Then $\oo_E^\times=\oo_F^\times(1+\p_E)\subseteq F+\p_E$ and therefore $\oo_E=\p_E\cup \oo_E^\times \subseteq F+\p_E$. It follows that $F+\oo_E=F+\p_E$. Assume that $\varpi_E^{-1}=x+y$ with $x\in F$ and $y\in\oo_E$.  Then, $x^{-1}=(\varpi_E^{-1}-y)^{-1}$ is a uniformizer of $E$ that lies in $F$ contradicting the assumption that the extension ramifies. It follows that $F+\oo_E  \subsetneq F+\p_E^{-1}$ and the lemma follows.
\end{proof}

Recall that the conductor of a non-trivial character $\xi$ of $E$ is the minimal integer $m$ such that $\xi$ is trivial on $\p_E^m$. For $a\in E$ let $\psi_a(x)=\psi(ax)$, $x\in E$, so that $\{\psi_a:a\in E\}$ is the group of all characters of $E$. 
\begin{lemma}\label{additive-conductor} for every integer $m$,
\begin{itemize}
\item if $E/F$ is unramified then there exists a character of $E$ trivial on $F$ of conductor $m$; 
\item if $E/F$ is ramified then there exists a character of $E$ trivial on $F$ of conductor $2m$. Furthermore all non trivial characters of $E$ which are trivial on $F$ have even conductor.
\end{itemize}
\end{lemma}
\begin{proof}
If $E$ is ramified over $F$ then it follows from the identity $F+\p_E^{2m+1}=F+\p_E^{2m} $ of Lemma \ref{inclusions} that the conductor of any non-trivial character of $E$ that is trivial on $F$ must be even. 

Note that the set of non-trivial characters of $E$ that are trivial on $F$ is $\{\psi_a:a\in F^\times\}$. Since the conductor of $\psi_a$ is $c-v$ where $c$ is the conductor of $\psi$ and $v$ is the $E$-valuation of $a$ and since the image of $F^\times$ under the $E$-valuation is $\Z$ if $E$ is unramified over $F$ and $2\Z$ otherwise, the lemma follows.
\end{proof}

If $A$ is an abelian locally compact totally disconnected group, we denote by $\widehat{A}$ its group of smooth characters. 

Recall that the conductor of $\chi\in \widehat{E^\times}$ is zero if $\chi$ is unramified (i.e., trivial on $\oo_E^\times$) and is the minimal positive integer $m$ such that $\chi|_{1+\p_E^m}=\trivchar$ otherwise. 
For $m\geq 0$, we denote by $\widehat{(E^\times/F^\times)}(m)$ the subset of $\widehat{E^\times}$ consisting of characters of $E^\times$ trivial on $F^\times$, and of conductor $\leq m$.

Note that, in fact, $\widehat{(E^\times/F^\times)}(m)$ is the group of characters of the finite group $E^\times/F^\times(1+\p_E^m)$.
\begin{lemma}\label{multiplicative-conductor}
Let $m\ge 1$ be an integer.
\begin{itemize}
\item If $E/F$ is unramified then $\widehat{(E^\times/F^\times)}(m-1)\subsetneq
\widehat{(E^\times/F^\times)}(m)$ and the 
natural map from $\widehat{(E^\times/F^\times)}(m)$ to $\widehat{(\frac{1+\p_E}{1+\p_F+\p_E^m})}$ induced by restriction is surjective. 
\item If $E$ is ramified over $F$ then $\widehat{(E^\times/F^\times)}(2m-1)\subsetneq\widehat{(E^\times/F^\times)}(2m)=
\widehat{(E^\times/F^\times)}(2m+1)$, and the 
natural map from $\widehat{(E^\times/F^\times)}(2m)$ to $\widehat{(\frac{1+\p_E}{1+\p_F+\p_E^{2m}})}$ induced by restriction is surjective. 
\end{itemize}
\end{lemma}
\begin{proof}
Note that $(1+\p_E)\cap F^\times(1+\p_E^m)=1+\p_F+\p_E^m$. It follows that $(1+\p_E)/(1+\p_F+\p_E^m)$ imbeds as a subgroup of the finite abeliean group $E^\times/F^\times(1+\p_E^m)$ and the surjectivity of the restriction map follows.
Furthermore, the trivial cahracter is the only unramified character of $E^\times$ that is trivial of $F^\times$ while the group $F^\times(1+\p_E)$ is strictly contained in $E^\times$ if $E/F$ is unramified (it consists of elements of even $E$-valuation).
In the unramified case, the inequality for $m=1$ follows.

To complete the proof of the lemma it is now enough to show that 
\begin{itemize}
\item $1+\p_F+\p_E^m \subsetneq 1+\p_F+\p_E^{m-1}$ if $E$ is unramified over $F$ and $m>1$;
\item $1+\p_F+\p_E^{2m+1}=1+\p_F+\p_E^{2m}\subsetneq 1+\p_F+\p_E^{2m-1}$ if $E$ is ramified over $F$. 
\end{itemize}
If $1+\p_F+\p_E^m = 1+\p_F+\p_E^{m-1}$ then it is easy to see that $F+\p_E^m =F+\p_E^{m-1}$. The inequalities are therefore immediate from Lemma \ref{inclusions}.
If $E/F$ is ramified, the same lemma shows that $\p_E^{2m}\subseteq F+\p_E^{2m+1}$. Write $y\in \p_E^{2m}$ as $y=a+z$ with $a\in F$ and $z\in \p_E^{2m+1}$. Then $a=y-z\in F\cap \p_E^{2m}\subseteq \p_F$.
Therefore, $1+\p_F+\p_E^{2m}=1+\p_F+\p_E^{2m+1}$. The rest of the lemma follows.
\end{proof}

\begin{corollary}\label{cor: multiplicative-conductor}
Assume that $\psi$ has conductor zero and let $m\ge 1$ be an integer. For any $c\in (F\cap \p_E^{-2m})\setminus \p_E^{-2m+1}$, there exists a character $\chi$ of $E^\times$, trivial on $F^\times$ and of conductor $2m$ such that $\psi(cx)=\chi(1+x)$ for all $x\in \p_E^m$.
\end{corollary}
\begin{proof}
Note that $\psi_c$ is a character of $E$ that is trivial on $F$ and has conductor $2m$. It therefore restricts to a character of $\p_F+\p_E^m$ that is trivial on $\p_F+\p_E^{2m}$. Also, $x\mapsto 1+x$ defines an isomorphism  $(\p_F+\p_E^m)/(\p_F+\p_E^{2m})\simeq (1+\p_F+\p_E^m)/(1+\p_F+\p_E^{2m})$. There is therefore a unique character $\xi$ of $(1+\p_F+\p_E^m)/(1+\p_F+\p_E^{2m})$ such that $\psi(cx)=\xi(1+x)$ for $x\in \p_F+\p_E^m$. Since any character on a subgroup of a finite abeliean group can be extended to the group, it follows from Lemma \ref{multiplicative-conductor} that there exists $\chi\in \widehat{(E^\times/F^\times)}(2m)$ such that $\psi(cx)=\chi(1+x)$ for $x\in \p_E^m$. This identity also implies that the conductor of $\chi$ equals $2m$.
\end{proof}

\begin{proposition}\label{prop: trivial central}
Let $n\geq 1$ and $\pi\in\Irr(n)$. If there is a constant $\gamma$ such that 
$\gamma(1/2,\pi,\chi,\psi)=\gamma$ for any distinguished character $\chi$ of $E^\times$, then $\gamma=1$ and 
the central character $c_\pi$ of $\pi$ is trivial on $F^\times$. 
\end{proposition}
\begin{proof}
The proof is an adaptation of \cite[Corollary 2.7]{MR3349305}. First, since
\[
\gamma(1/2,\pi,\chi,\psi_a)=c_\pi(a)\gamma(1/2,\pi,\chi,\psi), \ \ \ a\in F^\times, 
\]
by Lemma \ref{additive-conductor} we may assume that $\psi$ has conductor $0$. 
According to \cite[(2.7)]{MR787183} (to which we refer directly rather than to \cite[Proposition 2.6]{MR3349305} as the conductor of the additive character is $1$ in [ibid.], and 
$0$ here as in \cite{MR787183}), there is a positive integer $m(\pi)$ such that for $m\geq m(\pi)$, any character $\chi'$ of $E^\times$ of conductor $m\geq m(\pi)$ 
and any $c\in \p_E^{-m}$ which satisfies $\chi'((1+x))=\psi(cx)$ for all $x\in\p_E^{\lfloor (m+1)/2 \rfloor}$, we also have
\begin{equation}\label{JS} 
\gamma(s,\pi,\chi',\psi)=\epsilon(s,\pi,\chi',\psi)=c_\pi(c)^{-1}\epsilon(s,\mathbf{1},\chi',\psi)^n.
\end{equation}

Let $m$ be such that $2m\geq m(\pi)$ and take any $c\in \p_F^{-m}\setminus \p_F^{-m+1}$ if $E/F$ is ramified and any $c\in \p_F^{-2m}\setminus \p_F^{-2m+1}$ otherwise.
Then, $c\in (F\cap \p_E^{-2m})\setminus \p_E^{-2m+1}$ and by Corollary \ref{cor: multiplicative-conductor}, the character $\psi_c$ restricted to 
$\p_E^m$ is of the form $x\mapsto \chi(1+x)$ for $\chi$ a character of 
$E^\times/F^\times$ of conductor $2m$.

As $\epsilon(1/2,\mathbf{1}_{F^\times},\chi,\psi)=1$ for a distinguished character $\chi$, we deduce  from \eqref{JS} with $s=1/2$ and $\chi'=\chi$ that 
$c_\pi(c)^{-1}=\gamma$. As any element of $F^\times$ can be expressed as the quotient of two elements of $F$-valuation at most $d$, for any fixed $d$ (take
$d=-\lfloor m(\pi)/2 \rfloor$ if $E/F$ is ramified and $d=-m(\pi)$ otherwise), we 
deduce that ${c_\pi}_{|F^\times}=\mathbf{1}_{F^\times}$ and hence that $\gamma=1$.
\end{proof}


\section{The relative converse theorem for discrete series}\label{sec: discrete converse}

In this section we study the extent to which the property $\gamma(1/2,\pi,\tau,\psi)=1$ for `enough' distinguished representations $\tau$, implies that $\pi$ is distinguished.

For a class $\Irr_{\bullet}$ of representations in $\Irr$ let $\Irr_{\bullet}(n)=\Irr_{\bullet}\cap \Irr(n)$.

We will build our relative converse theorem for discrete series upon the following cuspidal relative converse theorem of \cite{MR3377749}, which itself builds upon and 
refines the results of \cite{MR2716711}. 
\begin{theorem}[\cite{MR3377749}]\label{thm: Ok}
Let $n>2$ and $\rho\in\Irr_\cusp(n)$ have a central character that is trivial on $F^\times$.
Then $\rho$ is distinguished if and only if $\gamma(1/2,\rho,\tau,\psi)=1$ for any distinguished $\tau\in \Irr_\ugen(n-2)$.
\end{theorem}

In fact, it follows from Proposition \ref{prop: trivial central} that the assumption on the central character can be replaced by the further requirement $\gamma(1/2,\rho,\tau,\psi)=1$ for any distinguished character $\tau\in \Irr(1)$.

We extend Theorem \ref{thm: Ok} to representations in $\Irr_\sqr$ and furthermore, show that it is enough to twist by distinguished tempered representations of $G_m$, $m\le n-2$, except for discrete series of $G_n$ of the form $St_2(\rho)$ with $\rho$ a cuspidal representation. For these discrete series we only obtain a weaker result, which says that it is enough to twist by tempered distinguished representations of 
$G_m$ for $m\le n$. We focus on the case $n>3$, as for $n=2,3$, it is known by \cite{MR1104321} and \cite{MR3377749} that the converse theorem with twists by distinguished characters of $E^\times$ holds for all generic unitary representations (and in fact for all generic representations when $n=2$ by \cite{MR2449011}). When $n>3$, Example \ref{optimal} shows that even for tempered representations, such a converse theorem cannot hold, so we restrict to discrete series. 

\textit{For the rest of this section, we assume $n>3$}. Let $c_\pi$ denote the central character of $\pi\in \Irr$. Our proof is based on three key observations that we first formulate as lemmas. The first is \cite[Proposition 3.0.3]{MR3377749}, together with the observation that $\lfloor n/2 \rfloor \leq n-2$ for $n\geq 3$.

\begin{lemma}\label{lem: HO}
Let $\delta \in \Irr_\sqr(n)\setminus \Irr_\cusp(n)$ satisfy $\delta^\gal\simeq\delta^\vee$ and 
${c_\delta}_{|F^\times}=\mathbf{1}$. If $\delta$ is not distinguished then there exists $k\leq n-2$ and a distinguished $\delta'\in \Irr_\sqr(k)$ such that $\gamma(1/2,\delta,\delta',\psi)=-1$.  \qed
\end{lemma}
The following observation will also prove itself useful.
\begin{lemma}\label{lem: is sym}
Let $\rho\in\Irr_\ucusp$ and $k\in\N$. If 
\[
\lim_{s\to k/2}\gamma(1/2+s,\St_k(\rho),\rho^\gal,\psi)\gamma(1/2-s,\St_k(\rho),\rho^\vee,\psi)=1
\]
then $\rho^\gal\simeq \rho^\vee$ and therefore $\St_k(\rho)^\gal\simeq \St_k(\rho)^\vee$. 
\end{lemma}
\begin{proof}
Since, as meromorphic functions of $s\in \C$ we have
\[
\gamma(1/2-s,\St_k(\rho),\rho^\vee,\psi)\gamma(1/2+s,\St_k(\rho^\vee),\rho,\psi^{-1})=1,
\]
it follows from the assumption that
\begin{equation}\label{eq: quot1}
\lim_{s\to \frac{k+1}2}\frac{\gamma(s,\St_k(\rho),\rho^\gal,\psi)}{\gamma(s,\St_k(\rho^\vee),\rho,\psi^{-1})}=1.
\end{equation}

By Lemma \ref{lem: cuspsupp} we have
\[
\gamma(s,\St_k(\rho),\rho^\gal,\psi)=\prod_{i=1}^k \gamma(s+\frac{k+1}2-i,\rho,\rho^\gal,\psi)
\] 
and 
\[  
\gamma(s,\St_k(\rho^\vee),\rho,\psi^{-1})=\prod_{i=1}^k \gamma(s+\frac{k+1}2-i,\rho^\vee,\rho,\psi^{-1}).
\]

It follows from Lemma \ref{lem: cusp L} that both $\gamma(s+\frac{k+1}2-i,\rho,\rho^\gal,\psi)$ and $\gamma(s+\frac{k+1}2-i,\rho^\vee,\rho,\psi^{-1})$ are holomorphic and non-zero at $s=\frac{k+1}2$ for all $i=1,\dots,k-1$ and $\gamma(s+\frac{1-k}2,\rho^\vee,\rho,\psi^{-1})$ has a simple pole at $s=\frac{k+1}2$. It therefore follows from \eqref{eq: quot1} that $\gamma(s+\frac{1-k}2,\rho,\rho^\gal,\psi)$ must also have a pole at $s=\frac{k+1}2$. By Lemma \ref{lem: cusp L} we deduce that $\rho^\gal\simeq \rho^\vee$. Since $\St_k(\rho)^\gal\simeq\St_k(\rho^\gal)$ and $\St_k(\rho)^\vee\simeq \St_k(\rho^\vee)$ the lemma follows.
\end{proof}



Let $\clss_{\le m}$ be the class of representations $\tau\in \cup_{k=1}^{m}\Irr_\temp(k)$ that are of one of the following two forms:
\begin{itemize}
\item $\tau\in \Irr_\sqr$ is distinguished or 
\item $\tau=\delta^\gal\times \delta^\vee$ where $\delta\in \Irr_{\usqr}$.
\end{itemize}
It is a class of distinguished tempered representations.

The key lemma allowing us to twist only by tempered representations is the following application of holomorphic continuation. 

\begin{lemma}\label{lem: mer cont}
Let $\pi\in \Irr(n)$ and $m\in \N$. 
\begin{enumerate}
\item\label{part: eps} If $\epsilon(1/2,\pi,\tau,\psi)=1$ for all $\tau\in \clss_{\le m}$ then $\epsilon(1/2,\pi,\tau,\psi)=1$ for every distinguished $\tau\in \cup_{k=1}^m\Irr(k)$. 
\item\label{part: gam} If $\pi\in \Irr_\temp$ and $\gamma(1/2,\pi,\tau,\psi)=1$ for all $\tau\in \clss_{\le m}$ then $\gamma(1/2,\pi,\tau,\psi)=1$ for every distinguished $\tau\in \cup_{k=1}^{m}\Irr_{<\frac12}(k)$.
\end{enumerate}
\end{lemma}
\begin{proof}
Let $\delta\in \Irr_\usqr(d)$ where $2d\le m$ and for $s\in\C$ let $\tau_s=\nu^s\delta^\gal\times \nu^{-s}\delta^\vee$. It follows from Lemma \ref{lem: mult} \eqref{part: Leps} that for $\Re(s)\ge 0$ we have
\[
\epsilon(1/2,\pi,\tau_s,\psi)=\epsilon(1/2+s,\pi,\delta^\gal,\psi)\epsilon(1/2-s,\pi,\delta^\vee,\psi).
\]
The right hand side is an entire function of $s\in \C$ and by assumption, it equals $1$ on the imaginary axis. By holomorphic continuation, it is therefore the constant function $1$. 

It follows that 
\[
\epsilon(1/2,\pi,\delta^\gal\times\delta^\vee,\psi)=1, \ \ \ \delta\in \Irr_\sqr, \  e(\delta)\ge 0. 
\]
Assume now in addition that $\pi$ is tempered. Then, by Lemmas \ref{lem: sqr L} and \ref{lem: mult} the function $\gamma(s,\pi,\tau,\psi)$ is holomorphic and non-zero at $s=1/2$ for all $\tau\in\Irr_{<\frac12}$. Furthermore, for $\abs{\Re(s)}<1/2$ we have
 \[
\gamma(1/2,\pi,\tau_s,\psi)=\gamma(1/2+s,\pi,\delta^\gal,\psi)\gamma(1/2-s,\pi,\delta^\vee,\psi)
\]
where, by the same lemmas, each factor on the right hand side is holomorphic for $\abs{\Re(s)}<1/2$. By assumption, the right hand side is identically $1$ for all $s$ in the imaginary axis and therefore, by holomorphic continuation, also whenever $\abs{\Re(s)}<1/2$. 

It follows that 
\[
\gamma(1/2,\pi,\delta^\gal\times\delta^\vee,\psi)=1, \ \ \ \delta\in \Irr_\sqr, \ 1/2>e(\delta)\ge 0.
\]

The two conclusions of the lemma now follow from Lemma \ref{lem: mult} and Proposition \ref{prop: dist std}.
\end{proof}
For discrete series we can now obtain a relative converse theorem with tempered twists.
\begin{theorem}\label{thm qsr}
Let $\delta=St_k(\rho)\in\Irr_\sqr(n)$. Suppose that $k\neq 2$, and that $\gamma(1/2,\delta,\tau,\psi)=1$ for all 
$\tau\in\clss_{\le n-2}$, then $\delta$ is distinguished. If $k=2$, and $\gamma(1/2,\delta,\tau,\psi)=1$ for all 
$\tau\in\clss_{\leq n}$, then $\delta$ is distinguished.
\end{theorem}
\begin{proof}
Write $\delta=\St_k(\rho)$ where $\rho\in\Irr_\cusp$. Notice that according to Proposition \ref{prop: trivial central}, the central 
character of $\delta$ is trivial on $F^\times$, so in fact $\delta\in\Irr_\usqr(n)$, i.e. $\rho\in\Irr_\ucusp$. If $k=1$ then the statement is that of Theorem \ref{thm: Ok}. 
Assume that $k\ge 3$. As $n>3$, this implies that if $k=3$, then $\rho$ is a representation of $G_m$ for some $m\geq 2$. 
In particular $\nu^s\rho^\gal\times\nu^{-s}\rho^\vee$ is a representation of $G_{n-(k-2)m}$ for any $s$, and $n-(k-2)m\leq n-2$.
 Hence by assumption, we have
\[
\gamma(1/2,\delta,\nu^s\rho^\gal\times\nu^{-s}\rho^\vee,\psi)=1, \ \ \ s\in i\R
\]
while, by Lemma \ref{lem: sqr L},
\[
\gamma(1/2,\delta,\nu^s\rho^\gal\times\nu^{-s}\rho^\vee,\psi)=\gamma(1/2+s,\delta,\rho^\gal,\psi)\gamma(1/2-s,\delta,\rho^\vee,\psi), \ \ \ s\in i\R.
\]
Since the right hand side is a product of two meromorphic functions for $s\in \C$ it follows by meromorphic continuation that
\[
\gamma(1/2+s,\delta,\rho^\gal,\psi)\gamma(1/2-s,\delta,\rho^\vee,\psi)=1
\]
as meromorphic functions of $s\in \C$ and in particular that 
\[
\lim_{s\to k/2}\gamma(1/2+s,\delta,\rho^\gal,\psi)\gamma(1/2-s,\delta,\rho^\vee,\psi)=1.
\]
It therefore follows from Lemma \ref{lem: is sym} that $\delta^\sigma\simeq\delta^\vee$. The theorem now follows from Lemma \ref{lem: HO}. 
The case $k=2$ is in fact the same, but in this case $\nu^s\rho^\gal\times\nu^{-s}\rho^\vee$ is also a representation of 
$G_n$, hence we are forced to allow twists by $\Gamma_{\leq n}$.
\end{proof}
\begin{example}\label{optimal}
For every $n\ge 4$ we now exhibit a non-distinguished representation $\tau\in \Irr_\temp(n)$ so that $\epsilon(1/2,\tau,\pi,\psi)=1$ for all distinguished $\pi\in \Irr$ and $\gamma(1/2,\tau,\pi,\psi)=1$ for all distinguished $\pi\in\Irr_{<\frac12}$. Let 
\[
\tau=\St_3(\eta)\times \eta\times \overbrace{\trivchar\times \cdots\times \trivchar}\limits^{n-4}.
\]
Then, $\tau$ is not distinguished e.g. by Proposition \ref{prop: dist std} and for any $\pi\in \Irr$ we have 
\[
\epsilon(1/2,\tau,\pi,\psi)=\epsilon(3/2,\eta,\pi,\psi)\epsilon(-1/2,\eta,\pi,\psi)\epsilon(1/2,\eta,\pi,\psi)^2\epsilon(1/2,\trivchar,\pi,\psi)^{n-4}.
\]
If $\pi$ is distinguished it follows from part \eqref{part: gamsqr}  of Lemma \ref{lem: L prop} that $\epsilon(3/2,\eta,\pi,\psi)\epsilon(-1/2,\eta,\pi,\psi)=1$, from part \eqref{part: epsqr} that $\epsilon(1/2,\eta,\pi,\psi)^2=1$ and from Theorem \ref{thm: eps triv} that $\epsilon(1/2,\trivchar,\pi,\psi)=1$.

Since $\tau$ is tempered, it further follows from Lemmas \ref{lem: sqr L} and \ref{lem: mult} that $L(s,\tau,\pi)$ is holomorphic at $s=1/2$ whenever $\pi\in\Irr_{<\frac12}$ and therefore, that $\gamma(1/2,\tau,\pi,\psi)=1$ by Lemma
\ref{lem: L prop} \eqref{part: gamhalf}.
\end{example}

\section{Triviality of root numbers-archimedean and global complements}\label{sec: complements}
Let $K$ be either $\R$ or $\C$ and in this section only,  let $G_n=\GL_n(K)$. 
By a representation of $G_n$ we mean an admissible, smooth Fr\'{e}chet representation of moderate growth (see \cite[11.6.8]{MR1170566}).
Let $\Irr(n)$ be the set of irreducible representations of $G_n$ and $\Irr=\cup_{n=1}^\infty \Irr(n)$. Fix a non-trivial character $\psi$ and for $\pi,\,\tau\in\Irr$ let 
$L(s,\pi,\tau)$ and $\epsilon(s,\pi,\tau)$ be the archimedean $L$ and $\epsilon$-factors defined via Langlands parameterization in terms of representations of the Weil group (see \cite{MR816396}).
The $\gamma$-factor $\gamma(s,\pi,\tau,\psi)$ is defined by \eqref{eq: gamma} where $\pi^\vee$ is the contragredient of $\pi$. 

Recall that the subset $\Irr_\sqr$ of essentially square integrable representation in $\Irr$ is contained in $\Irr(1)$ if $E=\C$ and in $\Irr(1)\cup \Irr(2)$ if $E=\R$. As in the $p$-adic case, let $\nu=\abs{\det}$ denoted the absolute value of the determinant on $G_n$ for any $n$ and for $\pi_1,\dots,\pi_k\in\Irr$ let $\pi_1\times \cdots \times \pi_k$ denote the associated normalized parabolic induction.
For $\delta\in \Irr_\sqr$ let $e=e(\delta)\in\R$ be such that $\nu^{-e}\delta$ is unitary. By the Langlands classification for any $\pi\in\Irr$ there exist $\delta_1,\dots,\delta_k\in\Irr_\sqr$ such that $e(\delta_1)\ge \cdots\ge e(\delta_k)$ and $\pi$ is the unique irreducible quotient of the standard module
\[
\lambda(\pi)=\delta_1\times\cdots\times \delta_k.
\]
If furthermore, $\tau\in\Irr$ has associated standard module $\lambda(\tau)=\delta_1'\times\cdots \times\delta_l'$ with $\delta_1',\dots,\delta_l'\in\Irr_\sqr$ then it essentially follows from the definitions that
\begin{equation}\label{eq: red to sqr}
L(s,\pi,\tau)=\prod_{i=1}^k\prod_{j=1}^l L(s,\delta_i,\delta_j') \ \ \ \text{and} \ \ \ \epsilon(s,\pi,\tau,\psi)=\prod_{i=1}^k\prod_{j=1}^l \epsilon(s,\delta_i,\delta_j',\psi).
\end{equation}

\subsection{Triviality of the archimedean root number for distinguished representations}\label{sec: arch}
Assume here that $E=\C$ and let $z^\gal=\bar z$, $z\in \C$ be complex conjugation and $\abs{x+iy}=x^2+y^2$, $x,\,y\in\R$ be the normalized absolute value.
Let $\psi(z)=e^{-2\pi(z-\bar z)}$. 
For a character $\xi$ of $\C^\times$ we further denote by
$L(s,\xi)$, $\epsilon(s,\xi,\psi)$ the local Tate $L$ and $\epsilon$ factors and let 
\[
\gamma(s,\xi,\psi)=\frac{\epsilon(s,\xi,\psi) L(1-s,\xi^{-1})}{L(s,\xi)}.
\]
We recall the functional equation
\begin{equation}\label{eq: fe tate}
\gamma(s,\xi,\psi)\gamma(1-s,\xi^{-1},\psi^{-1})=1=\epsilon(s,\xi,\psi)\epsilon(1-s,\xi^{-1},\psi^{-1}).
\end{equation}

A representation $\pi$ of $G_n$ is distinguished if there exists a continuous linear form $0\ne \ell\in\Hom_{\GL_n(\R)}(\pi,\trivchar)$.

Every character of $\C^\times$ is of the form
\[
\xi_{u,m}(z)=\abs{z}^u (\frac z{z^\gal})^m,\ \ \ u\in\C,\,2m\in\Z.
\]
The character $\xi_{u,m}$ is distinguished if and only if $u=0$ and $m\in \Z$. 

Let $\rho_k=\xi_{0,k}$, $ k\in \Z$ parameterise all distinguished characters.
It follows from the explicit computations of Tate's thesis that
\begin{equation}\label{eq: Tate comp}
\epsilon(1/2,\rho_k,\psi)=\gamma(1/2,\rho_k,\psi)=1,\  \ \ k\in \Z.
\end{equation}

As in Proposition \ref{prop: dist std}, based on the Langlands classification, the following is immediate from \cite[Theorem 1.2]{MR3404670}.
\begin{proposition}\label{prop: alex}
Let $\lambda$ be a distinguished standard module of $G_n$. Then there exist characters $\xi_1,\dots,\xi_t,\chi_1,\dots,\chi_s$ of $\C^\times$ with $e(\xi_1)\ge \cdots \ge e(\xi_t)\ge 0$ and $\chi_i$ distinguished for $i=1,\dots,s$ such that
\[
\lambda\simeq \xi_1^\gal \times\cdots \times \xi_t^\gal\times \chi_1 \times\cdots\times \chi_s\times  \xi_t^{-1}\times \cdots \times \xi_1^{-1}.
\] \qed
\end{proposition}
By the uniqueness of the Langlands data, as an immediate consequence we have (for generic representations this is \cite[Theorem 1.4]{MR3404670})
\begin{corollary}\label{cor: alex}
Let $\pi\in\Irr$ be distinguished then $\pi^\gal\simeq\pi^\vee$. \qed
\end{corollary}
We can now formulate the triviality of the local root number for distinguished representations.
\begin{theorem}\label{thm: arch rtnmb}
Let $\pi,\,\tau\in\Irr$ be distinguished. Then 
$\epsilon(1/2,\pi,\tau,\psi)=1$.
\end{theorem}
\begin{proof}
It follows from \eqref{eq: red to sqr}, \eqref{eq: Tate comp} and Proposition \ref{prop: alex} that it is enough to show that
\[
\epsilon(1/2,\xi^\gal,\psi)\epsilon(1/2,\xi^{-1},\psi)=1
\]
for all characters $\xi$ of $\C^\times$. Since $\psi^\gal=\psi^{-1}$ we have $\epsilon(1/2,\xi^\gal,\psi)=\epsilon(1/2,\xi,\psi^{-1})$ and the required identity follows from \eqref{eq: fe tate}.
\end{proof}

Let $\Gamma_\C(s)=2(2\pi)^{-s} \Gamma(s)$ where $\Gamma$ is the standard gamma function. It is a meromorphic function of $\C$ that is nowhere vanishing, has simple poles at $s\in\Z_{\le 0}$ and is holomorphic everywhere else.
Recall further from Tate's thesis that
\[
L(s,\xi_{u,m})=\Gamma_\C(s+u+\abs{m}).
\]

In particular, we have
\begin{itemize}
\item $L(s,\rho_k)$ is holomorphic at $s=1/2$ for all $k\in \Z$;
\item $L(s,\xi_{u,m})$ has a (simple) pole at $s=1/2$ if and only if $u+1/2+\abs{m}\in\Z_{\le 0}$.
\end{itemize}

As in the $p$-adic case we consider the following classes of representations in $\Irr$.
The tempered representations
\[
\Irr_\temp=\{\xi_1 \times\cdots\times \xi_k: \xi_i\in\Irr(1), \,e(\xi_i)=0,i=1,\dots,k\};
\]
and
\[
\Irr_{<\frac12}=\{\xi_1 \times\cdots\times \xi_k: \xi_i\in\Irr(1), \,\abs{e(\xi_i)}<1/2,i=1,\dots,k\}
\]
which is a set of generic irreducible representations containing all unitary, generic ones.
Based on the above, the following is obtained, as in Theorem \ref{thm: gamtriv}. We omit the proof.
\begin{theorem}
Let $\pi,\tau\in \Irr$ be distinguished representations. If $\pi\in\Irr_{<\frac12}$ and $\tau\in\Irr_\temp$, 
then $$\gamma(1/2,\pi,\tau,\psi)=1.$$
\end{theorem}

\subsection{The split case} Let $F$ be a local field and $E=F\oplus F$. Let $(x,y)^\gal=(y,x)$, $x,\,y\in F$. Fix a non-trivial character $\psi_0$ of $F$ and let $\psi(x,y)=\psi_0(x-y)$, $x,\,y\in F$. Set here, $G_n=\GL_n(E)=\GL_n(F)\times \GL_n(F)$ and let $\Irr(n)$ be the set of irreducible representations of $G_n$ and $\Irr=\cup_{n=1}^\infty \Irr(n)$. A representation $\pi \in \Irr$ is of the form $\pi=\pi_1\otimes\pi_2$ where $\pi_i$ are irreducible representations of $\GL_n(F)$, $i=1,2$. We say that $\pi$ is distinguished if $\pi_2\simeq \pi_1^\vee$.

For $\pi=\pi_1\otimes \pi_2,\,\tau=\tau_1\otimes \tau_2\in \Irr$ the local Rankin-Selberg $\epsilon$-factor is defined by
\[
\epsilon(s,\pi,\tau,\psi)=\epsilon(s,\pi_1,\tau_1,\psi_0)\epsilon(s,\pi_2,\tau_2,\psi_0^{-1}).
\]
As an immediate consequence of the standard functional equation for $\epsilon$-factors (Lemma \ref{lem: L prop} \eqref{part: gamsqr}) we have
\begin{lemma}\label{lem: split case}
Let $\pi,\,\tau\in\Irr$ be distinguished. Then 
\[
\epsilon(s,\pi,\tau,\psi)=1.
\]
\end{lemma}

\subsection{Triviality of the global root number for distinguished cuspidal representations}

Let $E$ be a number field and $\A=\A_E$ the ring of adeles of $E$. Fix a character $\psi$ of $\A/E$ and let $V_E$ be the set of places of $E$. For a pair of  
irreducible cuspidal automorphic representaions $\pi\simeq \otimes'_{v\in V_E}\pi_v$ of $\GL_n(\A)$ and $\tau\simeq \otimes'_{v\in V_E}\tau_v$ of $\GL_m(\A)$ consider the Rankin-Selberg $L$ and $\epsilon$ functions defined as the meromorphic continuation of the products
\[
L(s,\pi,\tau)=\prod_{v\in V_E} L(s,\pi_v,\tau_v) \ \ \ \text{and} \ \ \ \epsilon(s,\pi,\tau)=\prod_{v\in V_E} \epsilon(s,\pi_v,\tau_v,\psi_v).
\]
The first converges for $\Re(s)\gg 1$ and the second converges everywhere to an entire function, independent of $\psi$.
They satisfy the functional equation
\[
L(s,\pi,\tau)=\epsilon(s,\pi,\tau)L(1-s,\pi^\vee,\tau^\vee).
\]

Assume now that $E/F$ is a quadratic extension of number fields and let $\GL_n(\A_F)^1=\{g\in\GL_n(\A_F): \abs{\det g}=1\}$ (where $\abs{\cdot}$ is the standard absolute value on $\A_F$). We say that $\pi$ as above is distinguished, if there exists a cusp form $\phi$ in the space of $\pi$ such that 
\[
\int_{\GL_n(F)\bs \GL_n(\A_F)^1} \phi(h)\ dh\ne 0.
\]

For $v\in V_F$ let $E_v=F_v\otimes_F E$ and consider $F_v$ as a subfield of the algebra $E_v$ via $a\mapsto a\otimes 1$. 
Then $\A=\prod_{v\in V_F}' E_v$ and $\GL_n(\A)=\prod_{v\in V_F}' \GL_n(E_v)$. We may decompose accordingly, $\psi=\prod_{v\in V_F}\psi_v$ and $\pi\simeq \otimes'_{v\in V_F} \pi_v$. If $v$ is inert in $E$ and $w\in V_E$ is the unique place above $v$ then $E_v=E_w$, $\GL_n(E_v)=\GL_n(E_w)$ $\psi_v=\psi_w$ and $\pi_v=\pi_w$. If $v$ splits in $E$ and $w_1,\,w_2\in V_E$ lie above $v$ then $E_v\simeq E_{w_1}\oplus E_{w_2}\simeq F_v\oplus F_v$, $\GL_n(E_v)\simeq \GL_n(F_v)\times \GL_n(F_v)$, $\psi_v=\psi_{w_1}\otimes\psi_{w_2}$ and $\pi_v\simeq \pi_{w_1} \otimes \pi_{w_2}$ accordingly.  

Finally, clearly $E+\A_F$ is strictly contained in $\A_E$. Indeed, if $\gal$ is the Galois action for $E/F$, then $x-x^\gal\in E$ for all $x\in E+\A_F$. If $a\in E\setminus F$ is such that $a^2\in F$ and $w_0\in V_E$ lies over a place of $F$ that is split in $E$ let $x\in \A_E$ be such that $x_{w_0}=a$ and $x_w=0$ for all $w_0\ne w\in V_E$. Then $x-x^\sigma\not\in E$. Consequently, there exists a non-trivial character of $\A/E$ that is trivial on $\A_F$.

\begin{theorem}\label{thm: glb main}
For a pair of integers $n,m$ and irreducible cuspidal automorphic representaions $\pi\simeq \otimes'_{v\in V_E}\pi_v$ of $\GL_n(\A)$ and $\tau\simeq \otimes'_{v\in V_E}\tau_v$ of $\GL_m(\A)$, if both $\pi$ and $\tau$ are distinguished then 
\[
\epsilon(1/2,\pi,\tau)=1.
\]
\end{theorem}
\begin{proof}
Evidently, $\pi_v$ and $\tau_v$ are distinguished for all $v\in V_F$. Note that $\epsilon(1/2,\pi,\tau)=\prod_{v\in V_F} \epsilon(s,\pi_v,\tau_v,\psi_v)$ for any choice of non-trivial character $\psi$ of $\A/E$. We may choose $\psi$ to be trivial on $\A_F$ and then $\psi_v$ is a non-trivial character of $E_v$ that is trivial on $F_v$ for all $v\in V_F$. The theorem therefore follows from Theorems \ref{thm: eps triv} and \ref{thm: arch rtnmb} and Lemma \ref{lem: split case}.
\end{proof}


\def\cfudot#1{\ifmmode\setbox7\hbox{$\accent"5E#1$}\else
  \setbox7\hbox{\accent"5E#1}\penalty 10000\relax\fi\raise 1\ht7
  \hbox{\raise.1ex\hbox to 1\wd7{\hss.\hss}}\penalty 10000 \hskip-1\wd7\penalty
  10000\box7} \def\cprime{$'$} \def\cprime{$'$} \def\cprime{$'$}
  \def\cprime{$'$} \def\cprime{$'$} \def\cprime{$'$} \def\cprime{$'$}
  \def\cprime{$'$} \def\cprime{$'$} \def\cprime{$'$}
  \def\Dbar{\leavevmode\lower.6ex\hbox to 0pt{\hskip-.23ex \accent"16\hss}D}
  \def\cftil#1{\ifmmode\setbox7\hbox{$\accent"5E#1$}\else
  \setbox7\hbox{\accent"5E#1}\penalty 10000\relax\fi\raise 1\ht7
  \hbox{\lower1.15ex\hbox to 1\wd7{\hss\accent"7E\hss}}\penalty 10000
  \hskip-1\wd7\penalty 10000\box7}
  \def\cfudot#1{\ifmmode\setbox7\hbox{$\accent"5E#1$}\else
  \setbox7\hbox{\accent"5E#1}\penalty 10000\relax\fi\raise 1\ht7
  \hbox{\raise.1ex\hbox to 1\wd7{\hss.\hss}}\penalty 10000 \hskip-1\wd7\penalty
  10000\box7} \def\cftil#1{\ifmmode\setbox7\hbox{$\accent"5E#1$}\else
  \setbox7\hbox{\accent"5E#1}\penalty 10000\relax\fi\raise 1\ht7
  \hbox{\lower1.15ex\hbox to 1\wd7{\hss\accent"7E\hss}}\penalty 10000
  \hskip-1\wd7\penalty 10000\box7} \def\cprime{$'$}
  \def\Dbar{\leavevmode\lower.6ex\hbox to 0pt{\hskip-.23ex \accent"16\hss}D}
  \def\cftil#1{\ifmmode\setbox7\hbox{$\accent"5E#1$}\else
  \setbox7\hbox{\accent"5E#1}\penalty 10000\relax\fi\raise 1\ht7
  \hbox{\lower1.15ex\hbox to 1\wd7{\hss\accent"7E\hss}}\penalty 10000
  \hskip-1\wd7\penalty 10000\box7}
  \def\polhk#1{\setbox0=\hbox{#1}{\ooalign{\hidewidth
  \lower1.5ex\hbox{`}\hidewidth\crcr\unhbox0}}} \def\dbar{\leavevmode\hbox to
  0pt{\hskip.2ex \accent"16\hss}d}
  \def\cfac#1{\ifmmode\setbox7\hbox{$\accent"5E#1$}\else
  \setbox7\hbox{\accent"5E#1}\penalty 10000\relax\fi\raise 1\ht7
  \hbox{\lower1.15ex\hbox to 1\wd7{\hss\accent"13\hss}}\penalty 10000
  \hskip-1\wd7\penalty 10000\box7}
  \def\ocirc#1{\ifmmode\setbox0=\hbox{$#1$}\dimen0=\ht0 \advance\dimen0
  by1pt\rlap{\hbox to\wd0{\hss\raise\dimen0
  \hbox{\hskip.2em$\scriptscriptstyle\circ$}\hss}}#1\else {\accent"17 #1}\fi}
  \def\bud{$''$} \def\cfudot#1{\ifmmode\setbox7\hbox{$\accent"5E#1$}\else
  \setbox7\hbox{\accent"5E#1}\penalty 10000\relax\fi\raise 1\ht7
  \hbox{\raise.1ex\hbox to 1\wd7{\hss.\hss}}\penalty 10000 \hskip-1\wd7\penalty
  10000\box7} \def\lfhook#1{\setbox0=\hbox{#1}{\ooalign{\hidewidth
  \lower1.5ex\hbox{'}\hidewidth\crcr\unhbox0}}}
\providecommand{\bysame}{\leavevmode\hbox to3em{\hrulefill}\thinspace}
\providecommand{\MR}{\relax\ifhmode\unskip\space\fi MR }
\providecommand{\MRhref}[2]{%
  \href{http://www.ams.org/mathscinet-getitem?mr=#1}{#2}
}
\providecommand{\href}[2]{#2}

\end{document}